\documentclass[a4paper,12pt]{article}

\usepackage{amsthm}   \usepackage{amsfonts}     \usepackage{amssymb}
\usepackage{amsmath}  \usepackage{mathrsfs}     \usepackage{mathtools}
\usepackage{url}       \usepackage{fancyhdr}

\usepackage[all]{xy}    \usepackage{amscd}   \usepackage{epsfig}

\usepackage{array}   \usepackage{multirow}  \usepackage{booktabs,caption,fixltx2e}
\usepackage[flushleft]{threeparttable}

\usepackage{cases} 
\usepackage{ulem}  
\usepackage{xcolor}  

\usepackage[T1]{fontenc} 
\usepackage[utf8]{inputenc}
\usepackage{authblk}  
\usepackage{lineno}

\newtheorem{theorem}{Theorem}
\newtheorem{lemma}{Lemma}
\newtheorem{proposition}[theorem]{Proposition}
\newtheorem{corollary}[theorem]{Corollary}

\theoremstyle{definition}
\newtheorem{example}{Example}

\newtheorem*{case}{Case}
\newtheorem*{step}{Step}

\newcommand{\Z}{{\mathbb{Z}}}

\newcommand{\Aut}{{\mathrm{Aut}}}
\newcommand{\Inn}{{\mathrm{Inn}}}
\newcommand{\Out}{{\mathrm{Out}}}
\newcommand{\Mon}{{\mathrm{Mon}}}
\newcommand{\Gal}{{\mathrm{Gal}}}
\newcommand{\normal}{\trianglelefteq}

\newcommand{\la}{\langle}
\newcommand{\ra}{\rangle}

\begin{document}

\title{Nilpotent groups of class two which underly a unique regular dessin}
\author[1,2]{Kan Hu\thanks{hukan@zjou.edu.cn}}
\author[3,4]{Roman Nedela\thanks{nedela@savbb.sk}}
\author[1,2]{Na-Er Wang\thanks{wangnaer@zjou.edu.cn}}
\affil[1]{School of Mathematics, Physics and Information Science, Zhejiang Ocean University, Zhoushan, Zhejiang 316022, People's Republic of China}
\affil[2]{Key Laboratory of Oceanographic Big Data Mining \& Application of Zhejiang Province, Zhoushan, Zhejiang 316022, People's Republic of China}
\affil[3]{Faculty of Natural Sciences, Matej Bel University, Tajovsk\'eho 40, 974 01, Bansk\'a Bystrica, Slovak Republic}
\affil[4]{Institute of Mathematics and Computer Science, Slovak Academy of Sciences, Bansk\'a Bystrica, Slovak Republic}

\maketitle
\begin{abstract}
A dessin is an embedding of connected bipartite graph into an oriented closed surface. A dessin is regular if its group of colour- and orientation-preserving automorphisms acts transitively on the edges. In the present paper regular dessins with a nilpotent automorphism group are investigated, and attention are paid on those with the highest level of external symmetry. Depending on the algebraic theory of dessins and using group-theoretical methods, we present a classification of nilpotent groups of class two which underly a unique regular dessin.\\[2mm]
\end{abstract}


\section{Introduction}
  A  \textit{Bely\v{\i} function} is a non-constant meromorphic function $\beta\colon S\to\Sigma$ defined over a Riemann surface $S$ with at most three critical values  $0,$ $1$ and $\infty$ on the Riemann sphere $\Sigma$. Every Bely\v{\i} function $\beta$ determines a $2$-cell embedding of a $2$-colored bipartite graph on $S$ called a \textit{dessin}: The embedded bipartite graph is the preimage of the closed interval $[0,1]$ where the black vertices are $\beta^{-1}(0)$ and the white vertices are $\beta^{-1}(1)$; the faces are the components of $S\setminus\beta^{-1}[0,1]$. By Bely\v{\i}'s theorem~\cite{Belyi1979} a compact Riemann surface admits a Bely\v{\i} function (and hence a dessin) if and only if the surface, regarded as an algebraic curve,  is defined over the field $\mathbb{\bar Q}$ of algebraic numbers. The absolute Galois group $\Gal (\mathbb{\bar Q}/\mathbb{Q})$ acts on the coefficients of polynomials and rational functions defining the curves and Bely\v{\i} functions, and hence on the dessins. This provides a combinatorial approach to the absolute Galois group, as first observed by Grothendieck~\cite{Grothendieck1997}.

An \textit{automorphism} of a dessin $D$ is a permutation of its edges which preserves the graph incidence and vertex-colorings, and extends to a self-conformal homeomorphism of its supporting Riemann surface. The set of automorphisms of $D$ form the automorphism group $\Aut(D)$ of $D$ under composition. It is well known that the group $\Aut(D)$ acts semi-regularly on the edges. If this action is transitive, and hence regular, the dessin is called \textit{regular}.

Because the absolute Galois group acts faithfully on regular dessins~\cite{GJ2013} it is important to investigate regular dessins. The classification of regular dessins has been investigated by imposing certain conditions on the supporting surfaces, the embedded graphs or the underlying automorphism groups~\cite{Conder2012,CJSW2013,Hidalgo2013, HNW2014,Jones2010,Jones2013, MNS2012}.
In the present paper we focus on the third direction and consider regular dessins with nilpotent automorphism groups. Our main result is a classification of nilpotent groups of class two which underly a \textit{unique} regular dessin. The uniqueness implies that such dessins possess the highest level of external symmetry. The importance of such dessins lies in  the fact that they play the role of universal covers of other nilpotent regular dessins of the same nilpotence class~\cite[Section 6]{Jones2013}.


\section{External symmetries of dessins}
In this section we briefly outline the algebraic theory of dessins; see \cite{CS1988,JS1996} for more details.

Let $\Omega$ be the set of edges of a dessin $D$. Following the global orientation of the supporting surface of $D$ we obtain two permutations $\rho$ and $\lambda$ which successively permute the edges around the black and white vertices. The black  and white vertices correspond to the cycles of $\rho$ and $\lambda$ respectively, and the connectivity of the underlying graph of $D$ implies that the monodromy group $\Mon(D)$ of $D$ generated by $\rho$ and $\lambda$ is transitive on $\Omega$. It follows that each dessin $D$ determines a transitive permutation representation $\theta\colon F_2\to \Mon(D)$ of $F_2=\la X,Y\mid -\ra$, the free group of rank two. It is isomorphic to the action of $F_2$ (by right multiplication) on the cosets $Ng$ of a subgroup $N\leq F_2$. This subgroup $N$, the \textit{dessin-subgroup} associated with $D$, is the stabiliser in $F_2$ of an element in $\Omega$, and is uniquely determined up to conjugacy. The coverings $D_1\to D_2$ between dessins correspond to group inclusions $N_1\leq N_2$, and the automorphism group $\Aut(D)$ corresponds to the action of $N_{F_2}(N)/N$ on the cosets of $N$ where $N_{F_2}(N)$ is the normaliser of $N$ in $F_2$, acting by left multiplication.  In particular, regular dessins $D$ correspond to normal subgroups $N$ of finite index in $F_2$, in which case $\Aut(D)\cong F_2/N$.

An automorphism $\sigma$ of $F_2$ transforms a dessin $D$ to a dessin $D^{\sigma}$ by sending the dessin subgroup $N$ to $N^{\sigma}$. In particular, if $\sigma$ is an inner automorphism induced by an element $g\in F_2$, then $N^{\sigma}=g^{-1}Ng$, and hence $D\cong D^{\sigma}$. In this way the outer automorphism group $\Out(F_2):=\Aut(F_2)/\Inn(F_2)$,  the group of  \textit{dessin operations}, acts on the isomorphism classes of dessins. A dessin $D$ is said to \textit{possess an external symmetry $\sigma$} if $D\cong D^{\sigma}$ where $ \sigma\in\Aut(F_2) \backslash{\rm Inn}(F_2)$. This is equivalent to that the dessin subgroup $N$ of $D$ is $\sigma$-invariant, that is, $N^\sigma$ is conjugate to $N$. For instance, let $\sigma_1:X\to Y, Y\to X$ be the automorphism of $F_2$ transposing the generators, and let $\iota\colon X\to X^{-1}, Y\mapsto Y^{-1}$ be the automorphism of $F_2$ inverting the generators. If $D\cong D^{\sigma_1}$, then $D$ is called \textit{symmetric}, corresponding to an external symmetry transposing the vertex colors; if $D\cong D^{\iota}$, then $D$ is called \textit{reflexible}, corresponding to an external symmetry reversing the orientation of the supporting surface. A regular dessin which is invariant under all dessin operations will be called \textit{totally symmetric}.

In \cite{James1988} James showed that $\Out(F_2)\cong GL(2,\Z)$; in \cite{JP2010} Jones and Pinto proved that $\Out(F_2)=\la  \sigma_1\Inn(F_2),\sigma_2\Inn(F_2),\sigma_3\Inn(F_2)\ra$ where $\sigma_1$, $\sigma_2$ and $\sigma_3$ are automorphisms of $F_2$ defined by
\begin{align}
&\sigma_1\colon X\mapsto Y,\, Y\mapsto X,\label{Oper1}\\
&\sigma_2\colon X\mapsto Y, \,Y\mapsto X^{-1},\label{Oper2}\\
& \sigma_3\colon X\mapsto YX, Y\mapsto X^{-1}.\label{Oper3}
\end{align}
Hence we have
\begin{proposition}\label{TOTAL}
A regular dessin $D$ is totally symmetric if and only if $D^{\sigma_i}\cong D$ $(i=1,2,3)$.
\end{proposition}

Every regular dessin $D$ can be identified with a triple $(G,x,y)$ where $G=\Aut(D)$ and $x$ and $y$ are automorphisms of $D$ which generate, respectively, the stabilisers of a white vertex and an adjacent black vertex. The monodromy group and the automorphism group of $D$ are identified with the left and the right regular representations of $G$. In such an identification if $D$ possesses an external symmetry $\sigma$, then the assignment $\bar\sigma\colon x\mapsto \bar\sigma(x), y\mapsto\bar\sigma(y)$ extends to an automorphism of $G$.


\section{The universal cover}
In this section, following Jones~\cite{Jones2013} we give a precise construction of a totally symmetric dessin from a given group.

Let $D_i=(G_i,x_i,y_i)$ $(i=1,2)$  be two regular dessins, and let $N_i$ be the associated dessin subgroups in $F_2$. The subgroup $N_1\cap N_2$ of $F_2$, being the intersection of two normal subgroups of finite index in $F_2$, is also a normal subgroup of finite index in $F_2$. The regular dessin corresponding to $N_1\cap N_2$ is called the \textit{parallel product} of $D_1$ and $D_2$ \cite{Wilson1994} and is denoted by $D_1\vee D_2$.
The automorphism group of $D_1\vee D_2$ is isomorphic to a subgroup of $G_1\times G_2$  generated by $(x_1,x_2)$ and $(y_1,y_2)$~\cite{BN2001}.

Note that for a finite two-generated group $G$ there are finitely many non-isomorphic regular dessins $D$ with $\Aut(D)\cong G$. The set $\mathcal{R}(G)$ of isomorphism classes of regular dessins $D$ with $\Aut(D)\cong G$ corresponds bijectively to the set $\mathcal{N}(G)$ of normal subgroups $N$ of $F_2$ such that $F_2/N\cong G$, or to the set of orbits of $\Aut(G)$ acting on the generating pairs of $G$. Define
\begin{align}
K(G)=\bigcap_{N\in \mathcal{N}(G)} N.
\end{align}
Since $K(G)$ is the intersection of finitely many normal subgroups of finite index in $F_2$,  $K(G)$ is also a normal subgroup of finite index in $F_2$. Define
\begin{align}\label{UNIVERSAL}
U(G)=\bigvee_{D\in\mathcal{R}(G)} D\quad\text{and}\quad \bar G= F_2/K(G).
 \end{align}
Then $\Aut(U(G))\cong \bar G$.
By the construction, $K(G)$ is the unique normal subgroup of $F_2$ with quotient isomorphic to $\bar G$, and hence $U(G)$ is the unique regular dessin with automorphism group isomorphic to $\bar G$. In other words, the group $\Aut(\bar G)$ acts transitively on the generating pairs of $\bar G$. Since Galois conjugations preserve the automorphism group, the uniqueness implies that $U(G)$ is invariant under the action of the absolute Galois group, and is therefore defined over $\mathbb{Q}$.

\begin{example}\cite[Example 5.2]{Jones2013} Let $G=C_n$ be  the cyclic group of order $n$. Then there are precisely $\psi(n)$ regular dessins $D$ with $\Aut(D)\cong C_n$ where $\psi(n)$ is the Dedekind's totient function~\cite[Theorem 24]{HNW2014}; see also \cite[Example  3.1]{Jones2013}. Since $G$ is abelian of exponent $n$, $F_2' F_2^n\leq N$ for each $N\in \mathcal{N}(G)$ where $F_2'F_2^n$ is the group generated by the commutators and $n$th powers of elements of $F_2$. The minimality of $K(G)$ implies that $ K(G)=F_2'F_2^n$, so
\[
\bar G=F_2/F'F^n\cong C_n\times C_n.
\]
The regular dessin $U(G)$ is the $n$th degree Fermat dessin, corresponding to the standard embedding of $K_{n,n}$~\cite{Jones2010}.
\end{example}

\begin{example}
Let $G$ be a metacyclic $2$-group defined by the presentation
\[
G=\la g,h\mid g^8=h^8=1, h^g=h^5\ra.
\]
Let $P_1=(g,h)$, $P_2=(h,g)$ and $P_3=(g,gh)$. Then $P_i$ $(i=1,2,3)$ are generating pairs of $G$, giving us all three isomorphism classes of regular dessins with an automorphism group isomorphic to $G$~\cite[Example 2]{HNW2014}. Define
\begin{align*}
&S_1=\{X^8,Y^8, [X,Y]^2,[X,[X,Y]],[Y,[X,Y]], [X,Y]Y^4\}\subseteq F_2,\\
&S_2=\{X^8,Y^8,[X,Y]^2,[X,[X,Y]],[Y,[X,Y]],[X,Y]X^4\}\subseteq F_2,\\
&S_3=\{X^8,Y^8,[X,Y]^2,[X,[X,Y]],[Y,[X,Y]],[X,Y](X^{-1}Y)^4\}\subseteq F_2.
\end{align*}
Then the above dessins have dessin subgroups $S_i^{F_2}$ $(i=1,2,3)$, the normal closures of the sets $S_i$ in $F_2$. Moreover, let
\[T=\{X^8,Y^8,[X,Y]^2,[X,[X,Y]],[Y,[X,Y]]\}.\] Then $K(G)=T^{F_2}$, corresponding to a totally symmetric dessin $U(G)$ of type $(8,8,8)$ and genus $41$. The group $\bar G=F_2/K(G)$ has a presentation
\begin{align*}
\la x,y\mid x^8=y^8=z^2= [x,z]=[y,z]=1, z:=[x,y]\ra.
\end{align*}
Clearly, both $G$ and $\bar G$ are $2$-groups of class 2. It can be directly verified that the automorphism group of $\bar G$ acts transitively on the generating pairs of $\bar G$, and hence $\bar G$ underlies a unique regular dessin.
\end{example}

In what follows, we study the group $\bar G$ and the associated regular dessin in more detail. The following technical lemma will be useful.
\begin{lemma}\label{LIFT}
Let $m$ and $n$ be positive integers where $m|n$. Then for each number $s$, $1\leq s< m$, such that $\gcd(s,m)=1$, there is a number $s'$, $1\leq s'< n$, such that $\gcd(s',n)=1$ and $s'\equiv s\pmod{m}$.
\end{lemma}
\begin{proof}
If $n$ and $m$ contain the same prime factors, then since $\gcd(s,m)=1$, we also have $\gcd(s,n)=1$. In this case we can take $s':=s$. Otherwise, let $n'$ be the maximal factor of $n$ which is coprime to $m$. By the Chinese Remainder Theorem, there is a number $s'$, $1\leq s'<n$, which satisfies the following congruences:
\begin{equation*}
\begin{cases}
x\equiv1\pmod{n'}\\
x\equiv s\pmod{m}
\end{cases}
\end{equation*}
Clearly, $\gcd(s',n)=1$. Hence $s'$ is the desired number.
\end{proof}

\begin{proposition}\label{CHAR}
If a finite group underlies a unique regular dessin, then the dessin is totally symmetric with an underlying graph of multiplicity at most two.
\end{proposition}
\begin{proof}Assume that $G$ is a finite group which underlies a unique regular dessin $D=(G,x,y)$. Recall that dessin operations preserve the automorphism group. The uniqueness implies that the dessin $D$ is invariant under all dessin operations, and is therefore totally symmetric. In particular, $G$ has an automorphism $\bar\sigma_1$ transposing $x$ and $y$. So $o(x)=o(y)$. Assume $o(x)=n$. Recall that the multiplicity $m$ of the underlying graph is equal to the order of $\la x\ra\cap\la y\ra$. We have $m\mid n$ and $\la x\ra\cap\la y\ra=\la x^{n/m}\ra=\la y^{n/m}\ra$. Define $q=n/m$. Then there is a number $r$, $\gcd(r,m)=1$, such that $y^{q}=x^{qr}$. If $m\geq3$, then there is a number $s$, $1<s<m$, such that $\gcd(s,m)=1$. By Lemma~\ref{LIFT}, the number $s$ lifts to a number $s'$, $1<s'<n$, such that $\gcd(s',n)=1$ and $s'\equiv s\pmod{m}$. It follows that $G=\la x,y\ra=\la x^{s'},y\ra$. It follows from the uniqueness again that the generating pairs $(x,y)$ and $(x^{s'},y)$ belong to the single orbit under the action of $\Aut(G)$. Hence,  $x^{s'}$ and $y$ also satisfy the relation $y^q=x^{s'rq}$. So we get $x^{rq}=x^{s'rq}$, that is, $x^{qr(s'-1)}=1$. Since $\gcd(r,m)=1$ and $o(x)=mq$, we have $s'\equiv1\pmod{m}$.  Recall that $s'\equiv s\pmod{m}$, we have $s\equiv1\pmod{m}$. This is a contradiction to our choice of $s$.\end{proof}

The following example shows that the quaternion group underlies a unique regular dessin of multiplicity two.
\begin{example}\label{QUAT}
The quaternion group $Q_8$ has 24 distinct generating pairs. Recall that $\Aut(Q_8)\cong\mathrm{Sym}_4$, the symmetry group of degree 4. Since $\Aut(Q_8)$ acts semiregularly on the generating pairs of $Q_8$, the number of regular dessins with automorphism group isomorphic to $Q_8$ is equal to $24/|\Aut(Q_8)|=1$. The dessin is totally symmetric of type $(4,4,4)$ and genus $2$, corresponding to the $8$-gonal regular embedding of $K_{2,2}^{(2)}$ into the double torus.
\end{example}

The following result summarizes some properties of $\bar G$ when $G$ is solvable or nilpotent.

\begin{proposition}\cite[Section 5]{Jones2013}
Let $G$ be a $2$-generated group, and $\bar G$ defined by \eqref{UNIVERSAL}.
\begin{itemize}
\item[\rm(i)] If $G$ is solvable of derived length ${\rm dl(G)}$, then so is $\bar G$;
\item[\rm(ii)] If $G$ is nilpotent of class ${\rm c}(G)$, then so is $\bar G$.
\end{itemize}
\end{proposition}
\begin{proof}
The result follows from the fact that $\bar G$ is isomorphic to a subgroup of the direct product $G^r$ where $r=|\mathcal{R}(G)|$.
\end{proof}


\section{Classification }
In this section, we classify nilpotent groups of class 2 which underly a unique regular dessin.

The following decomposition theorem reduces the classification of nilpotent regular dessins to the classification of regular $p$-dessins.
\begin{proposition}\cite[Theorem 13]{HNW2014}\label{DECOM}
Every regular dessin with a nilpotent automorphism group $G$ is uniquely decomposed into a parallel product of regular dessins whose automorphism groups are the Sylow subgroups of $G$.
\end{proposition}

Recall that groups of class 1 are abelian. It is shown that an abelian $p$-group which underlies a unique regular dessin is isomorphic to $C_{p^a}\times C_{p^a}$ for some integer $a\geq0$~\cite[Theorem 24]{HNW2014}. In the remainder of the paper, we classify $p$-groups of class 2 which underly a unique regular dessin.

  The following prerequisites are assumed.
\begin{lemma}\label{BINO}\cite[Chapter III, Lemma 1.3]{Huppert1967}
Let $G$ be a nilpotent group of class 2, $x,y\in G$. Then
\[
[x^n,y]=[x,y^n]=[x,y]^n\quad\text{ and  }\quad (xy)^n=x^ny^n[y,x]^{n\choose 2},
\]
where $n\geq1$ is a positive integer.
\end{lemma}

\begin{lemma}\label{DERIVED}\cite[Chapter III, Lemma 1.11]{Huppert1967}
Let $G=\la x,y\ra$ be a group. Then $G'=\la [x,y]^g\mid g\in G\ra$.
\end{lemma}

\begin{lemma}\cite[Chapter III, Theorem 3.15]{Huppert1967}\label{BURNSIDE}
Let $G$ be a $p$-group and $\Phi(G)$ the Frattini subgroup of $G$. If $|G/\Phi(G)|=p^d,$ then every minimal generating set of $G$ contains exactly $d$ elements. In particular, $G=\la x_i\mid i=1,2,\ldots, d\ra$ if and only if $G/\Phi(G)=\la x_i\Phi(G)\mid i=1,2,\ldots, d\ra$.
\end{lemma}

The following theorem is the main result of the paper.
\begin{theorem}\label{MAIN}
A finite $p$-group $G$ of class two which underlies a unique regular dessin is isomorphic to one of the following groups:
\begin{itemize}
\item[\rm(i)] $p$ is odd and $1\leq b\leq a$:
\begin{align}
G=\la x,y\mid x^{p^a}=y^{p^a}=z^{p^b}=[x,z]=[y,z]=1,z:=[x,y]\ra.\label{GP1}
\end{align}
\item[\rm(ii)] $p=2$ and $1\leq b\leq a-1$:
\begin{align}
G=\la x,y\mid x^{2^a}=y^{2^a}=z^{2^b}=[x,z]=[y,z]=1, z:=[x,y]\ra.\label{GP2}\end{align}
\item[\rm(iii)] $p=2$ and $a\geq 2:$
\begin{align}
G=\la x,y\mid x^{2^a}=[x,z]=[y,z]=1, x^{2^{a-1}}=y^{2^{a-1}}=z^{2^{a-2}}, z:=[x,y]\ra.\label{GP3}
\end{align}
\end{itemize}
Moreover, the groups from distinct families, or from the same family but with distinct parameters, are pairwise non-isomorphic.
\end{theorem}
\begin{proof}
Assume that the group $G$ underlies a unique regular dessin $D=(G,x,y)$, so $\Aut(G)$ acts transitively on the generating pairs of $G$. By Proposition~\ref{CHAR}, $D$ is totally symmetric. It follows from the discussion at the end of Section 2 that the automorphisms $\sigma_i$ $(i=1,2,3)$ of $F_2$ defined by \eqref{Oper1}, \eqref{Oper2} and \eqref{Oper3} all induce automorphisms $\bar\sigma_i$ of $G$ given by
\begin{align*}
&\bar\sigma_1:x\mapsto y, y\mapsto x,\\
&\bar\sigma_2:x\mapsto y, y\mapsto x^{-1},\\
&\bar\sigma_3: x\mapsto yx, y\mapsto x^{-1}.
\end{align*}
 Using these properties, we first construct the group $G$ in the following two steps:

\begin{step}[1] Determination of the presentation of $G$.\par
Assume that $o(x)=p^a$ and $o(y)=p^{a'}$, $a,a'\geq 0$.  Define $z=[x,y]$ where $o(z)=p^b$. By Lemma~\ref{DERIVED}, $G'=\la z^g\mid g\in G\ra$. Since ${\rm c}(G)=2$, $1<G'\leq Z(G)$, and hence $G'=\la z\ra\cong C_{p^b}$ where $b\geq 1$.
 Since  $y=\bar\sigma_1(x)$, we have $o(x)=o(y)$, and hence $a=a'$. By induction we have $z^{p^a}=[x^{p^a},y]=1$, so $b\leq a$.

Let $N=\la z,x\ra$. Then $N\normal G$ is abelian and $G/N=\la yN\ra$ is cyclic. Assume $\la z\ra\cap\la x\ra= \la x^{p^c}\ra$ and $\la y\ra\cap N= \la y^{p^d}\ra$ where $0\leq c\leq a$ and $0\leq d\leq a$. Then there exist integers $r$, $s$ and $t$, where 
\[0\leq r\leq p^a-1,\quad 0\leq s\leq p^c-1, \quad 0\leq t\leq p^b-1,\]
 such that the following identities hold:
\begin{align}\label{EQUA}
x^{p^{c}}=z^{r}\quad\text{and}\quad y^{p^d}=x^{s}z^{t}
\end{align}
 Apply $\bar\sigma_1$, $\bar\sigma_2$ and $\bar\sigma_3$ to \eqref{EQUA} (use Lemma~\ref{BINO} if necessary) we get
\begin{align}
&y^{p^{c}}=z^{-r},\quad x^{p^d}=y^sz^{-t};&  \text{apply $\bar\sigma_1$ to \eqref{EQUA}} \label{EQUA1}\\
&y^{p^{c}}=z^{r},\qquad x^{-p^d}=y^sz^{t};&  \text{apply $\bar\sigma_2$ to \eqref{EQUA}} \label{EQUA2}\\
&y^{p^c}x^{p^c}=z^{r-{p^c\choose 2}},\quad y^sx^{s+p^d}=z^{-t-{s\choose2}}.& \text{apply $\bar\sigma_3$ to \eqref{EQUA}} \label{EQUA3}
\end{align}
By assumption $d$ is the smallest nonnegative integer $i$ such that $y^{p^i}\in N$, so the first relation in \eqref{EQUA1} implies that $d\leq c$. By induction we have
\[z^{p^d}=[x,y]^{p^d}=[x,y^{p^d}]\stackrel{\eqref{EQUA}}=[x,x^sz^t]=1.\]
Since $o(z)=p^b$, we get $b\leq d$. Summarizing the above inequalities we obtain
\begin{align}\label{INEQ}
1\leq b\leq d\leq c\leq a.
\end{align}

Moreover, by \eqref{EQUA} and \eqref{EQUA1}, substituting $z^{-r}$ and $z^{r}$ for $y^{p^c}$ and $x^{p^c}$ in the first relation of \eqref{EQUA3} we obtain that $z^{r-p^c(p^c-1)/2}=1$, which implies that
\begin{align}\label{CONG1}
r-p^c(p^c-1)/2\equiv0\pmod{p^b}.
\end{align}
Similarly, by \eqref{EQUA2} substituting $x^{-p^d}z^{-t}$ for $y^s$ in the second relation of \eqref{EQUA3}  we obtain that
\begin{align}\label{CONG2}
x^s=z^{-{s\choose2}}.
\end{align}
 Further, we obtain from the first relations in \eqref{EQUA1} and \eqref{EQUA2} that $z^{2r}=1$. So $p^b\mid 2r$. Therefore, if $p>2$, then $p^b\mid r$ ; if $p=2$, then either $2^b\mid r$ or $2^{b-1}\parallel r$. We distinguish 3 cases.

\begin{case}[i] $p>2$ and $p^b| r$.\par
Note that $\la x\ra\cap\la z\ra=z^r$. We have $\la x\ra\cap\la z\ra=1$. Since $\bar\sigma_1(\la x\ra\cap\la z\ra)=\la y\ra\cap\la z\ra$, we also have $\la y\ra\cap\la z\ra=1$.  So by \eqref{EQUA} we have $c=a$, and by \eqref{CONG2} we have $s\equiv0\pmod{p^a}$. In particular, the second relation of \eqref{EQUA} is reduced to $y^{p^d}=z^t$. Since $\la y\ra\cap\la z\ra=1$, we have $d=a$ and $t\equiv0\pmod{p^b}$. Consequently, $G$ has the presentation \eqref{GP1}.

\end{case}
\begin{case}[ii] $p=2$ and $2^b| r$.\par
In this case we also have $\la x\ra\cap\la z\ra=\la y\ra\cap\la z\ra=1$. So by \eqref{EQUA} we have $c=a$ and by \eqref{CONG2} we get $s\equiv0\pmod{2^a}$. It follows that the second relation of \eqref{EQUA} is reduced to $y^{p^d}=z^t$, so $d=a$ and $t\equiv0\pmod{2^b}$. Note that \eqref{CONG1} is reduced to $2^{a-1}(2^a-1)\equiv0\pmod{2^b}$, so $b\leq a-1$. Therefore, $G$ is defined by \eqref{GP2}.
\end{case}

\begin{case}[iii] $p=2$ and $2^{b-1}\parallel r$.\par
We have $\la x\ra\cap\la z\ra\cong\la y\ra\cap\la z\ra\cong C_2$. By the first relation of \eqref{EQUA} we have $c=a-1$. By \eqref{CONG1} we have $2^{b-1}(r'-2^{c-b}(2^c-1))\equiv0\pmod{2^b}$ where $r=2^{b-1}r'$, $r'$ is odd. This implies that $r'-2^{c-b}(2^c-1)$ is an even number. Hence $b=c$. By \eqref{INEQ} $b=d=c=a-1.$
So $y^{2^{a-1}}=z^{2^{a-2}}$ and $y^{2^{a-1}}=x^{s}z^t$. By \eqref{CONG2}, $s\equiv0\pmod{2^{a-1}}$. Recall that $0\leq s\leq 2^{a-1}-1$. Then we have $s=0$, implying $y^{2^{a-1}}=z^t$. Since $\la y\ra\cap\la z\ra\cong C_2$ and $o(z)=2^{a-1}$, we have $t=2^{a-2}$. Consequently we obtain the presentation \eqref{GP3}.
\end{case}
\end{step}

\begin{step}[2] Proof that $G$ underlies a unique regular dessin.\par
Equivalently, we shall prove that $\Aut(G)$ acts transitively on the generating pairs of $G$. Note that every element of $G$ can be written as the form $x^iy^jz^k$. Let $x_1=x^iy^jz^k$ and $y_1=x^ry^sz^t$. By Lemma~\ref{BURNSIDE}, $G=\la x_1,y_1\ra$ if and only if $G/\Phi(G)=\la \bar x^i\bar y^j, \bar x^r\bar y^s\ra$. Since $G/\Phi(G)\cong C_p\times C_p$ is an elementary $p$-group of rank 2, this is equivalent to that the matrix $\begin{pmatrix} i&j\\r &s\end{pmatrix}$ is invertible in $\Z_p$, that is,
\begin{align}\label{GEN}
is-jr\not\equiv0\pmod{p}.
\end{align}
It is sufficient to show that $x_1$ and $y_1$ satisfy the stated presentations in the respective case. Define $z_1=[x_1,y_1]$. We have
\[z_1=[x^iy^jz^k,x^ry^sz^t]=[x^iy^j,x^ry^s]=[x^i,x^ry^s][y^j,x^ry^s]=[x^i,y^s][y^j,x^r]=z^{is-jr}.\] By \eqref{GEN}, $p\nmid is-jr$, so $o(z_1)=o(z)$.  Clearly, $[x_1,z_1]=[y_1,z_1]=1$. Note that by Lemma~\ref{BINO} we have
\begin{align}\label{POWER}
(x^iy^jz^k)^n=(x^iy^j)^nz^{kn}=x^{in}y^{jn}[y^j,x^i]^{n\choose2}z^{kn}=x^{in}y^{jn}z^{kn-ij{n\choose2}},
\end{align}
where $n$ is a positive integer.  We distinguish 3 cases as before.

\begin{case}[i]By \eqref{POWER} we have  $x_1^{p^{a}}=x^{ip^{a}}y^{jp^{a}}z^{kp^{a}-ij{p^{a}\choose2}}=1.$ Since $\la x\ra\cap\la z\ra=1$ and $\la y\ra\cap\la x,z\ra=1$, by \eqref{GEN} and \eqref{POWER} we have
\begin{align*}
&x_1^{p^{a-1}}=(x^iy^j)^{p^{a-1}}z^{kp^{a-1}}=x^{ip^{a-1}}y^{jp^{a-1}}z^{kp^{a-1}-ij{p^{a-1}\choose2}}\neq1.
\end{align*}
 Therefore $o(x_1)=p^a$. Similarly, $o(y_1)=p^a$. So $x_1$ and $y_1$ satisfy the presentation~\eqref{GP1}.

\end{case}

\begin{case}[ii] By \eqref{POWER} we have $x_1^{2^a}=x^{i2^a}y^{j2^a}z^{k2^a-ij{2^a\choose2}}=1$. Since $\la x\ra\cap\la z\ra=1$ and $\la y\ra\cap\la x,z\ra=1$, by \eqref{GEN} and \eqref{POWER} we have
\begin{align*}
&x_1^{2^{a-1}}=(x^iy^j)^{2^{a-1}}z^{k2^{a-1}}=x^{i2^{a-1}}y^{j2^{a-1}}z^{k2^{a-1}-ij{2^{a-1}\choose2}}\neq1.
\end{align*}
 Therefore $o(x_1)=2^a$. Similarly, $o(y_1)=2^a$. It follows that $x_1$ and $y_1$ satisfy the presentation~\eqref{GP2}.

\end{case}

\begin{case}[iii] By \eqref{POWER} we have $x_1^{2^a}=x^{i2^a}y^{j2^a}z^{k2^a-ij{2^a\choose2}}=1$. Since $x^{2^{a-1}}=y^{2^{a-1}}=z^{2^{a-2}}$, we have
\begin{align*}
&x_1^{2^{a-1}}=x^{i2^{a-1}}y^{j2^{a-1}}z^{k2^{a-1}-ij{2^{a-1}\choose2}}=z^{(i+j+ij)2^{a-2}},\\
&y_1^{2^{a-1}}=x^{r2^{a-1}}y^{s2^{a-1}}z^{t2^{a-1}-rs{2^{a-1}\choose2}}=z^{(r+s+rs)2^{a-2}}.
\end{align*}
By \eqref{GEN}, $i+j+ij\equiv 1\pmod{2}$ and $r+s+rs\equiv1\pmod{2}$, so $o(x_1)=2^{a-1}o(z^{(i+j+ij)2^{a-2}})=2^a$ and $x_1^{2^{a-1}}=y_1^{2^{a-1}}=z_1^{2^{a-2}}$. Therefore $x_1$ and $y_1$ satisfy the presentation \eqref{GP2}.
\end{case}
\end{step}

 To finish the proof we need to show that the groups are uniquely determined by the parameters. This is easily seen from the following table which summarises the invariant types of the derived subgroups $G'$ and the abelianisations $G^{\mathrm{ab}}=G/G'$.
\[
\begin{array}{llll}
\text{Case} & G' &  G^{\mathrm{ab}} & \text{Condition}\\
\hline
{\rm(i)} &C_{p^a} & C_{p^a}\times C_{p^a} & 1\leq b\leq a\\
{\rm(ii)} &C_{2^{b+1}} & C_{2^{a-1}}\times C_{2^a} & 1\leq b\leq a-1\\
{\rm(iii)}&C_{2^{a-1}} & C_{2^{a-1}}\times C_{2^{a-1}} &a\geq 2\\
\end{array}
\]
\end{proof}

As a natural consequence of Theorem~\ref{MAIN} we have
\begin{corollary}Let $G$ be the groups from Theorem~\ref{MAIN}. Then the sizes of the group $G$ and its automorphism group $\Aut(G)$, and the type and genus of the associated regular dessin $U$ with $\Aut(U)\cong G$ are summarized as follows.
\[
\begin{array}{lllll}\label{TT}
Family & |G| &  |\Aut(G)| & \text{Type of $U$}&\text{Genus of $U$} \\
\hline
{\rm (i)} &p^{2a+b}& (p+1)(p-1)^2p^{4a+2b-3} & (p^a,p^a,p^a) &p^{a+b}(p^a-3)/2+1\\
{\rm(ii)} &2^{2a+b} & 3\cdot 2^{4a+2b-3} & (2^a,2^a,2^a)&2^{a+b-1}(2^a-3)+1\\
{\rm(iii)} & 2^{3a-4}& 3\cdot 2^{6a-9} &(2^a,2^a,2^a)&2^{2a-5}(2^a-3)+1\\
\end{array}
\]
\end{corollary}

\begin{proof}
From the proof of Theorem~\ref{MAIN}, we see that $|G|=p^d|N|=p^{b+c+d}$. The order $|G|$ of $G$ follows from substitution for $a$ and $c$ in the respective case. Since $\Aut(G)$ acts freely and transitively on the generating pairs of $G$, the order $|\Aut(G)|$ is equal to the number of generating pairs of $G$. The type and genus of the associated dessin $U$ are straightforward.\end{proof}

\section*{Acknowledgement}
The authors are grateful to the anonymous referee(s) for the helpful suggestions which have simplified the proof of Lemma~\ref{LIFT} and improved the presentation of the paper. The first and third author are supported by the following grants: Scientific Research Foundation(SRF) of Zhejiang Ocean Univesity and National Natural Science Foundation (NNSF: 60673096). The second author is supported by the following grants: APVV-0223-10, and the grant APVV-ESF-EC-0009-10 within the EUROCORES Programme EUROGIGA (Project GReGAS) of the European Science Foundation and the Slovak-Chinese bilateral grant APVV-SK-CN-0009-12.


\end{document}